\documentclass[english,reqno]{amsart}

\usepackage{graphicx,lscape}
\usepackage[mathscr]{euscript}
\usepackage{amsthm,amsmath,amsfonts,amssymb,eepic}
\usepackage{hyperref}
\usepackage[dvips]{color}
\usepackage{graphicx}
\usepackage{epsfig}
\usepackage[all]{xy}
\usepackage{amscd}
\usepackage{enumerate}
\usepackage{stmaryrd}
\usepackage{wasysym}

\newtheorem{theorem}{Theorem}[section]

\newtheorem{proposition}[theorem]{Proposition}

\theoremstyle{definition}

\newtheorem{definition}[theorem]{Definition}

\theoremstyle{remark}

\numberwithin{equation}{section}

\newcommand{\T}{\mathcal{T}}
\newcommand{\A}{\mathbf{A}}
\newcommand{\B}{\mathbf{B}}
\newcommand{\Op}{\mathcal{O}}
\newcommand{\W}{\mathcal{W}}
\newcommand{\C}{\mathbf{C}}
\newcommand{\Z}{\mathbf{Z}}
\newcommand{\G}{\widetilde{G}}
\newcommand{\F}{\acute{F}}

\newcommand{\lb}{\left(}
\newcommand{\rb}{\right)}

\begin{document}

\title{\textbf{A Bijective Proof for Reciprocity Theorem}}
\author{ShinnYih Huang, Alexander Postnikov}

\begin{abstract}
In this paper, we study the graph polynomial that records spanning rooted forests $f_G$ of a given graph. This polynomial has a remarkable reciprocity property. We give a new bijective proof for this theorem which has Pr\"ufer coding as a special case.
\end{abstract}
\maketitle
\section{\textbf{Introduction}}
A spanning tree $T$ in some graph $G$ is a connected acyclic subgraph of $G$ that includes all vertices in $V(G)$. Calculating the number $t( G ) $ of spanning trees for some graph $G$ is one of the typical questions we will ask. For example, when $G$ is a complete graph $K_n$, $t( K_n )  = n^{n-2}$. There are several methods to calculate $t( G ) $, such as the matrix-tree theorem and Pr\"ufer coding.

In this paper, we study some graph polynomial $f_G$ that records the spanning trees of the extended graph $\G$ of graph $G$. This polynomial can be used to compute the spanning tree of some complex graphs easily. For example, let $\Gamma = \Gamma \lb G;G_1,\ldots,G_k \rb $ be the graph that is obtained by substitution of graphs $G_1,\ldots,G_k$ instead of a vertices of a graph $G$. Then we can easily obtain $f_{\Gamma}$ by $f_G$ and $f_{G_i}$, for $1\leq i\leq k$.

In fact, the polynomial $f_G$ possess the remarkable property of reciprocity. A. Renyi [9] gives an inductive proof for this reciprocity theorem. I. Pak and A. Postnikov [1] also give an inductive proof. Throughout this paper, we present a new bijective proof for the reciprocity theorem. One interesting fact is that the map we used in the bijection is Pr\"ufer coding when $G$ is a complete graph.

This paper is organized as follows: In section 2, we define the graph polynomial $f_G$ to enumerate spanning trees in $\G$. In section 3, we show the reciprocity theorem for $f_G$ and defined some tools for the future bijective proof. In section 4, we define two maps $\phi$ and $\psi$ to show the bijection between $\A$ and $\B$. Finally, in section 5, we use this bijective coorespondence to prove the reciprocity theorem of $f_G$.

\section{\textbf{Graph Polynomials for Spanning Trees}}
Suppose that $G = ( V,E ) $ is a graph with vertices $1,\ldots,n$, where $|V| = n$. Let $0 \notin V$ and $\widetilde{V} := V \cup \{ 0 \}$. We say the \textit{extended graph} $\widetilde{G}$ of $G$ is a graph on the set $\widetilde{V}$ obtained by adding edges $\{0,v \}$ to $G$ for all vertices $v \in V$. Clearly, if $G$ is a complete graph $K_n$ with $n$ vertices, then $\widetilde{G}$ is a complete graph $K_{n+1}$ with $n+1$ vertices. We denote the set of all \textit{spanning trees} in $G$ as $\T_{G}$, i.e. all acyclic connected subgraphs in $G$ which contain all the vertices of $G$.

First of all, we assign variables $x_i$ to $i$, for all $1 \leq i \leq n$. For any spanning tree $T$ in $\T_{G}$, define a function $m( T ) $ associated to $T$:
\begin{equation}\label{etreep}
m( T )  = \prod_{v \in V} x_v^{\rho_{T}( v ) -1},
\end{equation}
where $\rho_{T}( v ) $ denotes \textit{degree} of the vertex $v$ in the tree $T$, i.e. the number of edges adjacent to the vertex $v$. 

Now, we set the graph polynomial $t_G$ to be,
$$t_G := \sum_{T\in \T( G ) }m( T ) .$$ 
Let us associate the variable $x$ to vertex $0$. Then, the graph polynomial $f_G$ of variables $x$ and $x_v$, for all $v \in V$ is defined as follows: 
\begin{equation}\label{egraphpf}
f_G := t_{\G} = \sum_{T\in \T_{\G}}m( T ) .
\end{equation}
We denote $V = \{1,\ldots,n \}$ and $f_G = f_G( x;x_1,\ldots,x_n ) $.

It is easy to see that the spanning trees in $\T_{\G}$ correspond to \textit{spanning rooted forests} in $G$, i.e. acyclic subgraphs in $G$ containing all vertices in $V$, with a root chosen in each component. In particular, the two polynomials $t_G$ and $f_G$ possess the following identity:
\begin{align}\label{etwographp}
t_G( x_1,\ldots,x_n )  \cdotp ( x_1+\cdots+x_n )  =  f_G( 0;x_1,\ldots,x_n ) .
\end{align}
An short proof for Eq.\eqref{etwographp} is provided in Igor Pak and A. Postnikov [1].

The graph polynomial $f_G$ has two important properties that allow us to compute the number of spanning rooted forests for certain graph. The first property is the composition of graphs. Let $G_1$ and $G_2$ be two graphs on disjoint sets of vertices, and $G_1 + G_2$ be the disjoint union of the graphs. We associate variable $x$ to the root $0$, variables $y_1,\ldots,y_{r_1}$ to the vertices of $G_1$, and variables $z_1,\ldots,z_{r_2}$ to the vertices of $G_2$. Then the following formula holds:
$$f_{G_1+G_2}( x;y_1\ldots,y_{r_1},z_1\ldots,z_{r_2} )  = x \cdotp f_{G_1}( x;y_1\ldots,y_{r_1} ) \cdotp f_{G_2}( x;z_1\ldots,z_{r_2} ) .$$
One can prove the above equation by some simple arguments. 

\section{\textbf{Reciprocity Theorem For Polynomials $f_G$} }
A graph $\overline{G} = ( V,\overline{E} ) $ is called the \textit{compliment} of some graph $G = ( V,E ) $ if $\overline{E} = \binom{V}{2}\backslash E$. That is to say, $e \in \overline{E}$ iff $e \notin E$. The graph polynomials $f_G$ possess the following \textit{reciprocity property}:
\begin{equation}\label{erecp}
f_G{( x;x_1,\ldots,x_n ) } = ( -1 ) ^{n-1} \cdotp f_{\overline{G}}( {-x-x_1-\cdots-x_n;x_1,\ldots,x_n} ) .
\end{equation}
The case that $x_1 = \cdots = x_n = 1$ for \eqref{erecp} was found by S. D.
Bedrosian [2] and A. Kelmans.

Before we give the bijective proof for Eq.\eqref{erecp}, we first introduce some notation.

First of all, let $\F_{G}$ be a spanning tree of some extended graph $\G$ with root $0$ and vertices $1,\ldots,n$ so that $F_{G}$ is a spanning rooted forest of $G$. It is easy to show that for any vertex $u$ of $G$, there is a unique path from $u$ to root $0$. Therefore, we can assign a direction to every edge in $\F_{G}$ such that each arrow points toward the root $0$. This implies that every vertex $u \neq 0$ has outdegree 1. For convention, in this paper, when we say graphs $\F_G \in \T_{\G}$ or $F_G$, we always consider it as a directed graph, and thus for every $u \neq 0$, there is a unique directed edge $(u,v) \in E(\F_G)$. In addition, a vertex $u$ is the \textit{child} of vertex $u_1$ if there is a directed path from $u$ to $u_1$ in $\T_{\G}$.

Secondly, we say that a \textit{valid pair} of some tree $\F_{K_n}$ is a pair $( u,v )  \in F_{K_n}$, and $\Z_{G,\F_{K_n}}$ is a subset of valid pairs of $\F_{K_n}$ such that
\begin{equation}\label{evalidpairs}
\Z_{G,\F_{K_n}} = \{( u,v ) : ( u,v )  \notin E(\overline{G}), ( u,v )  \in  E(F_{K_n}) \}  .
\end{equation}
Now, given a subset $\C$ of all valid pairs not in $\Z_{G,\F_{K_n}}$, we define an \textit{operational set} $\Op_{G,\F_{K_n},\C}$ as follows:
\begin{equation}\label{eopers}
\Op_{G,\F_{K_n},C} = \C \cup \Z_{G,\F_{K_n}}.
\end{equation}
One can see that for a spanning tree $\F_{K_n}$ and graph $G \in K_n$, there could be many possible operational sets. An example is in figure \ref{fig:1}.
\begin{figure}[h]
\centering
\includegraphics[width=9.5cm]{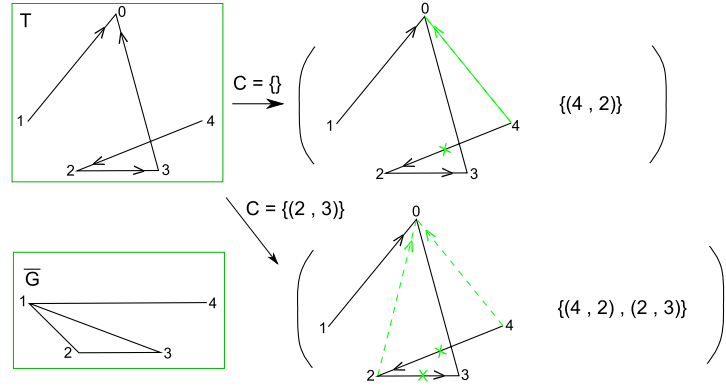}
\caption{For $\F_{K_n}$ and $\overline{G}$ as above, we have two possible operational sets for $\F_{K_n}$. (The green marks are the graph after we apply all the pair in the operation sets to $\F_{K_n}$.) }
\label{fig:1}
\end{figure}

Now, for any $\F_{\overline{G}}$, suppose its induced subgraph $F_{\overline{G}}$ in $K_n$ has $k$ connected components. We say a \textit{weight sequence} $\W_{\F_{\overline{G}}}$ of $\F_{\overline{G}}$ is
\begin{equation}\label{eweight}
\W_{\F_{\overline{G}}} = ( w_1,\ldots,w_{k-1}  ), 
\end{equation}
where $w_j \in \{0,1,\ldots,n\}, \text{ for } 1 \leq j \leq k-1$. By convention, if $k=1$, we set $W_{T_{\widetilde{\overline{G}}}}$ to be empty. Therefore, there are $(n+1)^{k-1}$ possible weight sequences for spanning tree $\F_{\overline{G}}$ that has $k$ connected compoenents in $F_{\overline{G}}$.

Given a graph $G\in K_n$, let $\A$ be the set of all possible pairs $\lb \F_{K_n}, \Op_{G,\F_{K_n},C} \rb $ and $\B$ be the set of all possible pairs $\lb \F_{\overline{G}},\W_{\F_{\overline{G}}} \rb $. In the following section, we show a bijection between $\A$ and $\B$.

\section{\textbf{Bijection Between $\A$ to $\B$}}
\noindent
Suppose that $G$ is a graph with $n$ vertices labeled $1,\ldots,n$ where each vertex $i$ is associated to a variable $x_i$, for $1\leq i\leq n$. For the root in the extended graph, we assign variable $x$ to root $0$. We first construct a map $\phi$ from $\A$ to $\B$.

\begin{definition} Given a pair $\left( \F_{K_n}, \Op_{G,\F_{K_n},C} \right)  \in A$, the map $\phi$ outputs a pair $(\F,\W) $ and is defined as follows:
\vspace{5pt}

Let $S$ be the set of vertices $u$ in $\F_{K_n}$, where the directed edge $(u,v) \in E(\F_{K_n})$ is a pair in $\Op_{G,\F_{K_n},C}$ or $v = 0$. Construct an empty sequence $\W$ and a graph $\F$ which is a duplicate of $\F_{K_n}$.

\vspace{5pt}

WHILE $|S| > 1$,

\begin{description}
\item[1] Suppose there is a leaf $u' \neq 0$ in $\F_{K_n}$ such that the edge $(u',v') \in E(\F_{K_n})$ is not in $S$. We remove $u'$ and  $( u',v' ) $ from $\F_{K_n}$.

\item[2] Repeat step 1 until every leaf $u \neq 0$ in $\F_{K_n}$ is also in $S$. Let $M$ to be the set of all these vertices.

\item[3] Delete the largest vertex $u^{*}$ in $M$ and the directed edge $( u^{*},v^{*} ) $ in $\F_{K_n}$. We set $S$ to be $S \backslash \{u^{*} \}$, and add $v^{*}$ to the end of the sequence $\W$.

\item[4] Remove edge $( u^{*},v^{*} ) $ and add edge $( u^{*},0)$ to $\F$.
\end{description}
\vspace{5pt}

RETURN $(  \F,\W   ) $.
\vspace{5pt}
\end{definition}

\noindent
An example of this algorithm is in figure \ref{fig:2}. In the following proposition, we prove that $\phi$ is well-defined.
\begin{figure}[h]
\centering
\includegraphics[width=9.5cm]{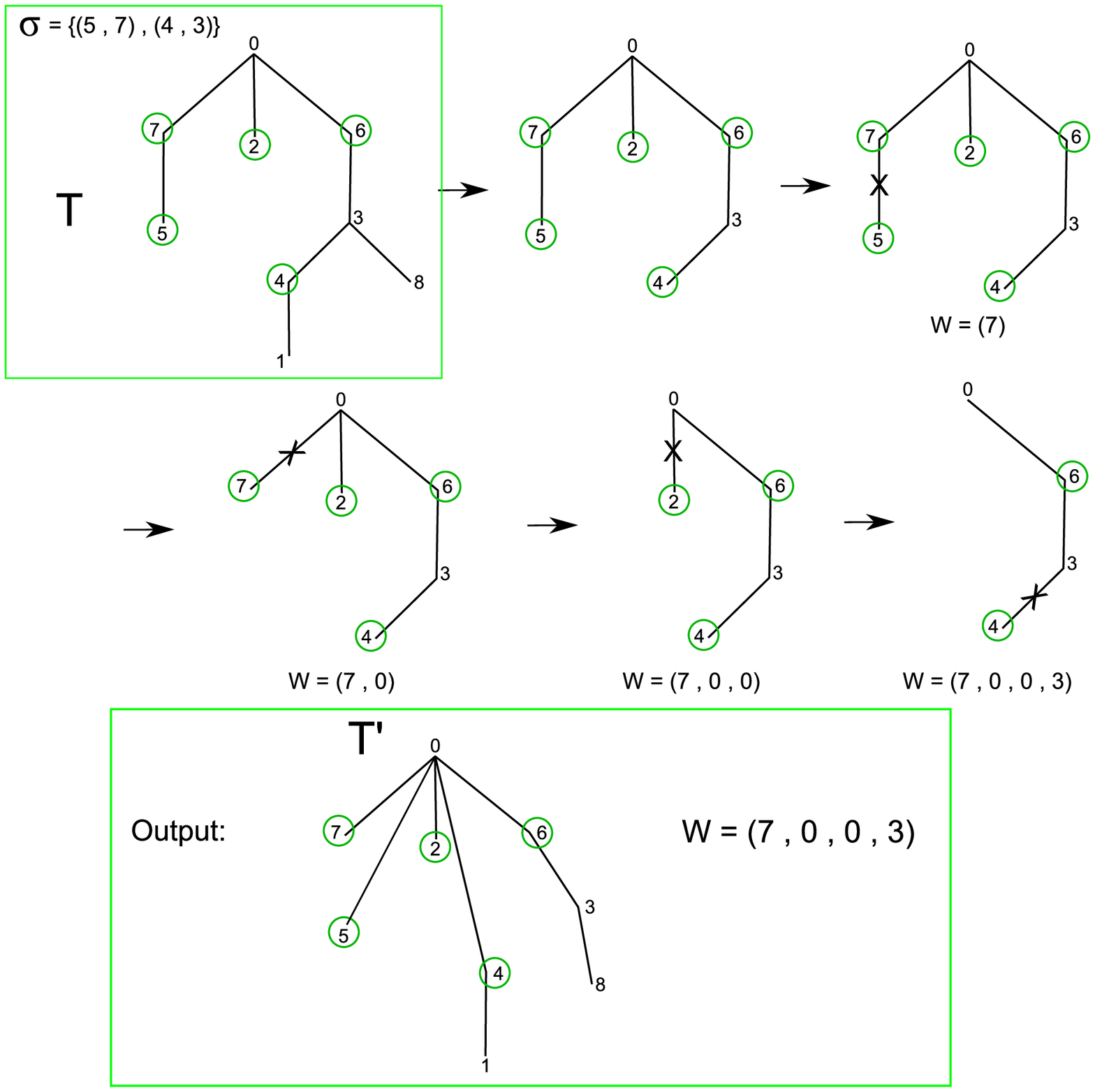}
\caption{Input: $T = \F_{K_n}$ and $\Op = \Op_{G,\F_{K_n},C} = \{( 5,7 ) ,( 2,3 )  \}$, Output: $T' = \F_{\overline{G}}$ and $\W = \W_{\F_{\overline{G}}} = \{7,0,0,3 \}$}
\label{fig:2}
\end{figure}

\begin{proposition}
The map $\phi$ is a well-defined map from $\A$ to $\B$.
\end{proposition}

\begin{proof}
It is easy to see that all the steps in WHILE loop work. Now, we show that $\F$ is a spanning tree of $\widetilde{K_n}$ after each step 4. We proceed this by induction.

 Initially, $\F = \F_{K_n}$ is a tree. Suppose that at some step 4, we delete edge $( u^{*},v^{*} )$ and add edge $( u^{*},0 ) $ to the spanning tree $\F \in \T_{\widetilde{K_n}}$. Furthermore, since for any vertex $u\neq 0$, $u$ and root $0$ is connected in  graph $\F$, it remains connected after we change some edge $( u^{*},v^{*} )$ to edge $( u^{*},0 ) $. Since $|E(\F)| = n$, $\F$ is always a spanning tree of $\widetilde{K_n}$ after any step 4.

Now, from \eqref{eopers}, we know that $\Z_{G,\F_{K_n}} \in \Op_{G,\F_{K_n},C}$ and all the edges $( u,v )$ in the operational set $\Op_{G,\F_{K_n},C}$ became $( u,0 )$ in the output graph $\F$. Thus, every edge in $E(F)$ is also in $E(\overline{G})$, and $\F$ is a spanning tree of $\widetilde{\overline{G}}$.

Finally, we show that $W$ is a weight sequence of $\F$. Clearly, $S$ is the set of all roots in the spanning rooted forest $F$. Since the WHILE loop ends when $|S| = 1$, there are totally $|S| -1 $ elements added to the sequence $W$. Consequently, $W$ satisfies the length requirement in Eq.\eqref{eweight}.

The above arguments tell us that $(  \F,\W   ) \in B$ as desired.
\end{proof}
\noindent
We now give a map $\psi$ from $\B$ to $\A$.

\begin{definition} Given a pair $\left( \F_{\overline{G}},\W_{\F_{\overline{G}}} \right) \in \B$, the map $\psi$ outputs $\left( \F^{*}, \Op \right) $ and is defined as follows: 

Assume that the forest $F_{\overline{G}}$ has $k$ connected components and the associated weight sequence $\W_{\F_{\overline{G}}} = ( w_1,\ldots,w_{k-1}  )$. Create a tree $\F^{*} = \F_{\overline{G}}$, sequence $\W_{\F^{*}} = \W_{\F_{\overline{G}}}$, and an empty set $\Op$. Let $R$ be the set of roots in  $F_{\overline{G}}$.

\vspace{5pt}

WHILE the length of $\W_{\F^{*}}$ is larger than 0.

\begin{description}
\item[1] We choose the first element $w$ in the sequence $\W_{\F^{*}}$. Let $u$ be the largest vertex in $R$ such that $w_i$ is not $u$ nor a child of $u$ in $\F^{*}$, for any $w_i$ in $\W_{\F^{*}}$. Delete the element $w$ from the sequence $\W_{\F^{*}}$ and $u$ from the set $R$.

\item[2] Remove the edge $( u,0 ) $ and add the edge $( u,w ) $ to the graph $\F^{*}$. If $w \neq 0$, we add pair $( u,w ) $ to the set $\Op$, i.e. $\Op = \Op \cup \{( u,w ) \}$.
\end{description}
\vspace{5pt} 

RETURN $( \F^{*},\Op  )  $.
\vspace{5pt} 
\end{definition}

An example of this mapping $\psi$ is in figure \ref{fig:3}. In the following lemma, we prove that $\psi$ is well-defined.

\begin{figure}[h]
\centering
\includegraphics[width=9.5cm]{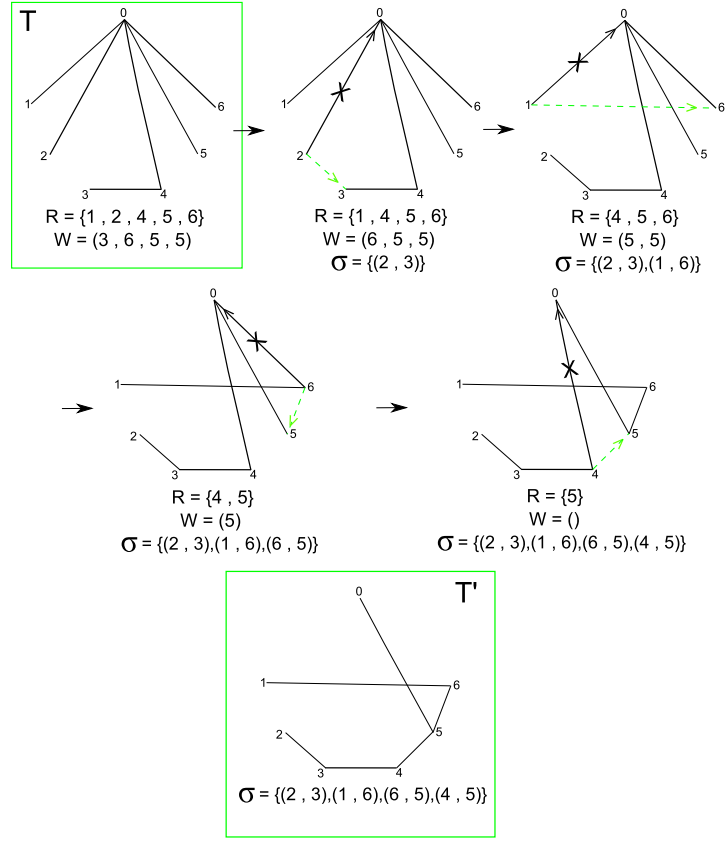}
\caption{Input: $T = \F_{\overline{G}}$ and $W = \W_{\F_{\overline{G}}}  = \{3,6,5,5 \}$, Output: $T'=\F^{*}$ and operational set $\sigma = \Op = \{( 2,3 ) ,( 1,6 ) ,( 6,5 ) ,( 4,5 )  \}$. (R is the set of current roots.) }
\label{fig:3}
\end{figure}

\begin{proposition}
The map $\psi$ is a well-defined map from $\B$ to $\A$.
\end{proposition}

\begin{proof}
We first show that at any stage, the set $R$ and graph $\F^{*}$ satisfy the following properties:
\begin{enumerate}
\item $\F^{*}$ is a spanning tree of $\widetilde{K_n}$, i.e. $F^{*}$ is a sapnning rooted forest of $K_n$.
\item $R$ is the sets of roots of forest $F^{*}$.
\end{enumerate}

We proceed by induction on the number of loops. Initially, $R$ is the set of all the roots in forest $F_{\overline{G}}$, and $\W_{\F^{*}}$ is a sequence of length $k-1=|R|-1$. Moreover, at each step 1, we remove an element in $\W_{\F^{*}}$ and an element in $R$. Thus, the length of sequence $\W_{\F^{*}}$ is always $|R|-1$. 

Now, suppose at some stage, we have that properties (1) and (2) hold and sequence $\W_{\F^{*}} = \{w'_1,\ldots,w'_{k_1-1} \}$, where $k_1 = |R|$. During step 1, since there are $k_1$ connected components in $F^{*}$, there exists at least one connected component that contains no elements in $\W_{\F^{*}}$. Consider the compoenent with the largest root $u$ that meets this condition. It is not hard to see that for any $1\leq i \leq k_1-1$, $w'_i$ is not $u$ nor a child of $u$. Consequently, step 1 works. 

For step 2, by the choice of vertex $u$, we have $w'_i$ and $u$ are not connected in $F^{*}$. Suppose $\F^{*}$ becomes cyclic after we delete edge $(u,0)$ and add edge $(u,w'_1)$ to this graph. This implies that there is a cycle containing edge $(u,w'_1)$. It is not possible since vertices $u$ and $w'_1$ would be connected in $F^{*}$ before we add edge $(u,w'_1)$.

The above arguments show that after step 1 and 2, $\F^{*}$ remains acyclic, and is a spanning tree of $\widetilde{K_n}$. Futhermore, after step 2, since $u$ is no longer a root, $R$ remains as the set of all roots in $F^{*}$. As a result, properties (1) and (2) always hold.

Finally, we need to show that $( v,v' )  \in \Op$, for every directed edge $( v,v' )  \notin E(\overline{G})$ and $( v,v' )  \in E(F^{*})$. Clearly, $\F^{*}$ is obtained from $\F_{\overline{G}}$ by a series of removing and adding edges in step 2. If edge $( v,v' )  \notin E(\overline{G})$, then $(v,v') \notin E(\F_{\overline{G}})$. Therefore, edge $(v,v')$ is added to graph $\F^{*}$ in some step 2, and $( v,v' )  \in \Op$. This implies that $(\F^{*},\Op) \in A$ as desired. 
\end{proof}

\begin{theorem}
The two maps $\phi$ and $\psi$ define a bijective correspondence between sets $\A$ and $\B$.
\end{theorem}

\begin{proof}
We have shown that $\phi$ and $\psi$ are well-defined. The remaining task is to prove that $\phi$ is the inverse map of $\psi$. 

Given a pair $\left( \F_{K_n},\Op_{G,\F_{K_n},C} \right) \in \A $, we apply the map $\phi$ and obtain an output $(\F,\W  ) \in \B$. Suppose that during the map $\phi$, we record the largest vertex $u^{*}$ in every step 3 into a sequence $U$ in order. It is easy to see that $|U| = |\W| = |S|-1$, where $S$ is the original set before the WHILE loop in map $\phi$. Let $|S| = k$, and we set $U = \{u_1,\ldots,u_{k-1}   \}$ and $\W = \{w_1,\ldots,w_{k-1}   \}$. Thus, for any $1\leq j \leq k-1$, $( u_j,w_j ) $ is the directed edge removed from $\F_{\overline{G}}$ in step 3 in the $j$-th WHILE loop. 

Now, let us apply the map $\psi$ on pair $(\F,W) \in \B$, and denote the output pair by $\left( \F '_{K_n},\Op_{G,\F '_{K_n},C_1} \right)\in \A$. Therefore, initially, $R = S$ is the set of roots of forest $F$. Our goal is to prove that 
\begin{equation}\label{e1}
\F_{K_n} = \F '_{K_n}\text{ and }\Op_{G,\F_{K_n},C_1} = \Op_{G,\F '_{K_n},C_1}.
\end{equation}
We record the vertex $u$ we picked in every step 1 in the map $\psi$ and get a sequence $U' = \{u'_1,\ldots,u'_{k-1}\}$ in order. Clearly, if $U$ and $U'$ are the same sequence, Eq.\eqref{e1} holds since every move in step 2 in $\psi$ will be the reverse move in step 4 in $\phi$. 

Before we show that $U = U'$, we first prove the following property:

\begin{enumerate}
\item In the $i$-th WHILE loop of the map $\phi$, where $1\leq i \leq k-1$, consider the graph $\F_{K_n}$ after step 2. Then for any $u$ in that current set $S$, it is not a leaf in $\F_{K_n}$ iff there exists some $w_{i_1}$, where $i \leq i_1 \leq k-1$, such that $w_{i_1}$ is $u$ or a child of $u$.
\end{enumerate}

If $u$ is not a leaf in $\F_{K_n}$, then there must be a vertex $u'$ in current set $S$ that is child of $u$. Consider the vertex $w'$ which edge $(u',w')$ is in $E(\F_{K_n})$. Consequently, $w' \in \{w_i,\ldots,w_{k-1} \}$ is vertex $u$ or child or $u$. By some easy arguments, one can see that the reverse statement is true, and thus prove property (1).

We now show $U=U'$ by induction on the index $i$, where $1 \leq i \leq k-1$. When $i = 1$, clearly, from (1), we know that $u'_1$ is a leaf in $\F_{K_n}$. By the choice of $u_1$, we have $u'_1 \leq u_1$. On the other hand, since $u'_1$ is the largest element in $S$ that no element in $\W$ is $u'_1$ or child of $u'_1$, we have $u_1 \leq u'_1$. As a result, $u_1 = u'_1$. 

Secondly, suppose for $i$ from 1 to $r-1$, where $r \leq k-1$, we have $u_i = u'_i$. That is to say, the set $S$ and $R$ in the $r$-th WHILE loop of map $\phi$ and $\psi$ are the same. When $i = r$, from (1) and the choice of $u'_r$, we have that both $u_r \in S$ and $u'_r \in R = S$ are the largest vertex $z$ such that no element $w \in \{ w_r,\ldots,w_{k-1} \}$ is $z$ or child of $z$. Consequently, $u_r = u'_r$.  

By induction, we can prove that $U$ and $U'$ are the same sequence. Therefore, Eq.\eqref{e1}  holds and $\psi$ is the inverse map of $\phi$. Finally, this shows us that the two maps $\phi$ and $\psi$ define a coorespondence relation between sets $\A$ and $\B$. 
\end{proof}

In particular, consider the case that $G = K_n$. Since $\overline{G}$ is empty, we have that every valid pair $( u,v ) $ in $\F_{K_n}$ is not in $\overline{G}$. Therefore, for every spanning tree $\F_{K_n}$ in $\widetilde{K_n}$, there is only one possible operational set $\Op_{K_n,\F_{K_n},C} = Z_{K_n,\F_{K_n}}$. In addition, there is only one spanning tree $\F_{\overline{G}}$ which is the graph with every vertex connected to root 0. Consequently, for every pair $( \F_{\overline{G}},\W_{\F_{\overline{G}}} )  \in B$, we have that $|\W_{\F_{\overline{G}}}| = n-1$. That is to say, every element in $\B$ is associated to a sequence of length $n-1$. One can easily see that the map $\phi$ now is a prufer coding for spanning trees in $K_{n+1}$ and therefore, prufer coding is a special case for this bijection.
\section{A New Proof of The Reciprocity Theorem}
In this section, we show how to use this bijection to prove the reciprocity theorem. 

\begin{theorem}
Let $G$ be a graph on the set of vertices $\{1,\ldots, n   \}$. Then
\begin{equation}\label{e2}
f_G{( x;x_1,\ldots,x_n ) } = ( -1 ) ^{n-1} \cdotp f_{\overline{G}}( {-x-x_1-\cdots-x_n;x_1,\ldots,x_n} ). 
\end{equation}
\end{theorem}

\begin{proof}
First of all, we show that 
\begin{equation}\label{e3}
( -1 ) ^{n-1} \cdotp f_{G}( {x;x_1,\ldots,x_n} )  = f_{G}( {-x;-x_1,\ldots,-x_n} ). 
\end{equation}
If we can show that the degree of every monomial in $f_{G}( {x;x_1,\ldots,x_n} ) $ is $n-1$, then Eq.\eqref{e3}  will be true. Note that each monomial in $f_{G}( {x;x_1,\ldots,x_n} ) $ corresponds to some spanning tree $\F_{\widetilde{K_n}}$ of $\widetilde{K_n}$, and we have
\begin{align}\displaystyle
\deg\left(m\left(\F_{\widetilde{K_n}}\right)\right) = \sum_{v\in \{0,\ldots,n \}} \big{( }\deg( v )  - 1 \big{ ) } &= \displaystyle\sum_{v\in \{0,\ldots,n \}} \deg( v )  - ( n+1 ) \\
 &= 2|E| - ( n+1 )  = n-1.
\end{align}
This implies that Eq.\eqref{e3} is true. 

Now, we show that 
\begin{equation}\label{e4}
f_G{( x;x_1,\ldots,x_n ) } = f_{\overline{G}}( {x+x_1+\cdots+x_n;-x_1,\ldots,-x_n} ). 
\end{equation} 
Consider some spanning tree $\F_{K_n}$ of $\widetilde{K_n}$ associated to a monomial $x^{d}x_1^{d_1}\cdots x_n^{d_n}$ in polynomial $f_G$ and an operational set $\Op_{G,\F_{K_n},C}$ for $\F_{K_n}$. Let us apply the map $\phi$ on $\left( \F_{K_n}, \Op_{G,\F_{K_n},C} \right)$. Denote the output pair by $\left(\F_{\overline{G}},\W_{\F_{\overline{G}}}\right) \in \B$, where sequence $\W_{\F_{\overline{G}}} = ( w_1,\ldots,w_{k-1} ) $, and $k$ is the number of connected components in $F_{\overline{G}}$. Moreover, the contribution of graph $\F_{\overline{G}}$ in the polynomial $f_{\overline{G}}$ is
\begin{equation}\label{e5}
( x+x_1+\cdots+x_n ) ^{k-1}( -x_1 ) ^{\deg( v_1 ) -1}\cdots( -x_n ) ^{\deg( v_n ) -1},
\end{equation}
where $\deg( v_i ) $ is the degree of vertex $i \neq 0$ in $\F_{\overline{G}}$. We associate the pair $\left( \F_{\overline{G}},\W_{\F_{\overline{G}}} \right) $ to the monomial 
$$x_{w_1}\cdots x_{w_{k-1}} ( -x_1 ) ^{\deg( v_1 ) -1}\cdots( -x_n ) ^{\deg( v_n ) -1}$$
in \eqref{e5}, where $x_0 = x$ and $x_{w_j}$ is the variable corresponding to vertex $w_j$, for $1\leq j \leq k-1$. Clearly, $x_{w_1}\cdots x_{w_{k-1}}$ is a monomial in $( x+x_1+\cdots+x_n ) ^{k-1}$. By the choice of $\W_{\F_{\overline{G}}}$ shown in section 3, we have that the set $\B$ and set of all monomials in $f_{\overline{G}}( {x+x_1+\cdots+x_n;-x_1,\ldots,-x_n} )$ have a bijective coorespondence.

It is easy to show that the monomial for the pair $\left( \F_{K_n},\Op_{G,\F_{K_n},C} \right)$ is the monomial associated to the pair $\left( \F_{\overline{G}},\W_{\F_{\overline{G}}} \right) $ with several sign changes, where the number of sign changes is $\displaystyle \sum_{i = 1}^{n}(\deg( v_i )-1)$. That is to say, we have 
\begin{equation}\label{e6}
x^{d}x_1^{d_1}\cdots x_n^{d_n} = (-1)^{l} \cdotp x_{w_1}\cdots x_{w_{k-1}} x_1^{\deg( v_1 ) -1}\cdots x_n^{\deg( v_n ) -1},
\end{equation}
where $l = \displaystyle \sum_{i = 1}^{n}(\deg( v_i )-1) = n- \deg(v_0)$.

Now, suppose that $\F_{K_n} \in \T(\widetilde{G})$. Since every valid pair in $\F_{K_n}$ is not in graph $\overline{G}$, the only operational set for $\F_{K_n}$ is $\Z_{G,\F_{K_n}}$. In addition, the output spanning tree $\F_{\overline{G}}$ is the extended graph of empty graph. Therefore, the only pair $\left( \F_{K_n},\Z_{G,\F_{K_n}} \right)\in \A $ for $\F_{K_n}$ is mapped to a monomial in $( x+x_1+\cdots+x_n ) ^n$. This implies that the coefficient of the monomial associated to $\F_{K_n}$ is $1$ in $f_{\overline{G}}( {x+x_1+\cdots+x_n;-x_1,\ldots,-x_n} ) $.

Secondly, if $\F_{K_n} \notin \T(\widetilde{G})$, then there is an edge $( u,v )  \in E(F_{K_n})$ such that $( u,v )  \in E(\overline{G})$. For every operational set $\Op_{G,\F_{K_n},C} $ for $\F_{K_n}$, we consider the two operational sets:
\begin{equation}\label{e7}
 \Op_1 = \Op_{G,\F_{K_n},C} \cup \{ ( u,v )  \}, \text{ and }\Op_2 = \Op_1 \backslash \{ ( u,v )  \} 
\end{equation}
Clearly, $\Op_1$ and $\Op_2$ are both operational sets for $\F_{K_n}$. Denote the output pair for $( \F_{K_n}, \Op_1 ) $ as $( \F_1,\W_1 ) $ and the output pair for $( \F_{K_n}, \Op_2 ) $ as $( \F_2,\W_2 ) $ in the map $\phi$. From Eq.\eqref{e6}, one can see that the monomials associated to the two pairs $( \F_1,\W_1 ) $ and $(\F_2,\W_2)$ are the same. Moreover, the degrees of root $0$ in $\F_1$ and $\F_2$ are differ by 1. Consequently, by \eqref{e6}, the summation of the coefficients of the monomial associated to $( \F_1,\W_1 ) $ and $(\F_2,\W_2)$ is 0. Finally, because we can pair up all the operational sets for $\F_{K_n}$ by \eqref{e7}, the contribution of the monomial for $\F_{K_n}$ in $f_{\overline{G}}( {x+x_1+\cdots+x_n;-x_1,\ldots,-x_n} ) $ is 0. 

From the above argument, we conclude that the only monomials left in $f_{\overline{G}} $ after cancellation of coefficients are the monomials in $f_G{( x;x_1,\ldots,x_n ) }$. Moreover, each monomial in $f_{G}$ has coefficient 1 in  $f_{\overline{G}}( {x+x_1+\cdots+x_n;-x_1,\ldots,-x_n} ) $. As a result, we have that $f_G{( x;x_1,\ldots,x_n ) } = f_{\overline{G}}( {x+x_1+\cdots+x_n;-x_1,\ldots,-x_n} ) $, and Eq.\eqref{e2} holds as desired.
\end{proof}

\end{document}